\documentclass[a4paper,10pt]{amsproc}
\usepackage{amsfonts, amsmath, amssymb, graphicx, mathrsfs}
\theoremstyle{plain}

\newtheorem{theorem}{Theorem}[section]

\theoremstyle{definition}
\newtheorem{definition}[theorem]{Definition}

\newtheorem{remark}[theorem]{Remark}
\newtheorem{counter example}[theorem]{Counter Example}

\newtheorem{corollary}[theorem]{Corollary}
\newtheorem{example}[theorem]{Example}
\newtheorem{question}[theorem]{Question}
\numberwithin{equation}{section}

\begin{document}

\title[Recent progress in Rings and Subrings of Real Valued Measurable Functions]{Recent progress in Rings and Subrings of Real Valued Measurable Functions}

%Information for first author:

%Information for second author (if needed):
\author[S. Acharyya]{Soumyadip Acharyya}
\address{Department of Mathematics, Physical, and Life Sciences, Embry-Riddle Aeronautical University Worldwide, Daytona Beach, FL 32114, USA}
\email{acharyys@erau.edu}

\author[S.K.Acharyya]{Sudip Kumar Acharyya}
\address{Department of Pure Mathematics, University of Calcutta, 35, Ballygunge Circular Road, Kolkata 700019, West Bengal, India}
\email{sdpacharyya@gmail.com}

\author[S. Bag]{Sagarmoy Bag}
\address{Department of Pure Mathematics, University of Calcutta, 35, Ballygunge Circular Road, Kolkata 700019, West Bengal, India}
\email{sagarmoy.bag01@gmail.com}
\thanks{The second author thanks the NBHM, Mumbai-400 001, India, for financial support}

\author[J. Sack]{Joshua Sack} % changed order of authors to alphabetical order
\address{Department of Mathematics and Statistics, California State University Long Beach, 1250 Bellflower Blvd, Long Beach, CA 90840, USA}
\email{joshua.sack@csulb.edu}

 %conditionally complete lattice; $P$-space%}
\thanks {}

%                                                                                                                           %
%         Please use the current 2010 Mathematics Subject Classification:             %
%         http://www.ams.org/mathscinet/msc/                                                        %
%         http://www.zentralblatt-math.org/msc/en/                                                 %
%%%%%%%%%%%%%%%%%%%%%%%%%%%%%%%%%%%%%%%%%%%%%%%%%%%

\thanks {}

\maketitle

\begin{abstract}
Two separated realcompact measurable spaces $(X,\mathcal{A})$ and $(Y,\mathcal{B})$ are shown to be isomorphic if and only if the rings $\mathcal{M}(X,\mathcal{A})$ and $\mathcal{M}(Y,\mathcal{B})$ of all real valued measurable functions over these two spaces are isomorphic. It is furthermore shown that any such ring $\mathcal{M}(X,\mathcal{A})$, even without the realcompactness hypothesis on $X$, can be embedded monomorphically in a ring of the form $C(K)$, where $K$ is a zero dimensional Hausdorff topological space. It is also shown that given a measure $\mu$ on $(X,\mathcal{A})$, the $m_\mu$-topology on $\mathcal{M}(X,\mathcal{A})$ is 1st countable if and only if it is connected and this happens when and only when $\mathcal{M}(X,\mathcal{A})$ becomes identical to the subring $L^\infty(\mu)$ of all $\mu$-essentially bounded measurable functions on $(X,\mathcal{A})$. Additionally, we investigate the ideal structures in subrings of $\mathcal{M}(X,\mathcal{A})$ that consist of functions vanishing at all but finitely many points and functions 'vanishing at infinity' respectively. In particular, we show that the former subring equals the intersection of all free ideals in $\mathcal{M}(X,\mathcal{A})$ when $(X,\mathcal{A})$ is separated and $\mathcal{A}$ is infinite. Assuming $(X,\mathcal{A})$ is locally finite, we also determine a pair of necessary and sufficient conditions for the later subring to be an ideal of $\mathcal{M}(X,\mathcal{A})$.  
\end{abstract}

\section{Introduction}
In what follows $(X,\mathcal{A})$ stands for a nonempty set $X$ endowed with a $\sigma$-algebra $\mathcal{A}$ of subsets of $X$. Such a pair is called a measurable space. A function $f:X\mapsto \mathbb{R}$ is called $\mathcal{A}$-measurable if for any $\alpha\in \mathbb{R}, f^{-1}(\alpha,\infty)$ is a member of $\mathcal{A}$. It is well known that the family $\mathcal{M}(X,\mathcal{A})$ of all real valued $\mathcal{A}$-measurable functions on $X$ makes a commutative lattice ordered ring with unity if the relevant operations are defined pointwise on $X$. An $\mathcal{A}$-filter ($\mathcal{A}$-ultrafilter) on $X$ stands for a filter (ultrafilter) on $X$ consisting of members of $\mathcal{A}$. By an intermediate ring of measurable functions, we mean a subring $\mathcal{N}(X,\mathcal{A})$ of the ring $\mathcal{M}(X,\mathcal{A})$ which contains $\mathcal{M}^*(X,\mathcal{A})$ of all bounded $\mathcal{A}$-measurable functions on $X$. In \cite{ABS}, a recently communicated article, we have initiated a kind of duality between the ideals (maximal ideals) of  $\mathcal{M}(X,\mathcal{A})$ as well as the intermediate rings and a collection of   appropriately defined $\mathcal{A}$-filters ($\mathcal{A}$-ultrafilters) on $X$. It is also realized that the structure spaces of all the intermediate rings are one and the same. Incidentally this common structure space is seen to be homeomorphic to the aggregate of all $\mathcal{A}$-ultrafilters on $X$ equipped with the Stone-topology. The structure space of a commutative ring $R$ with unity as usual means the set of all maximal ideals of $R$ equipped with the hull-kernel topology.

The present article is indeed a continuation of our study on the above interaction between ideals (maximal ideals) of rings and subrings of $\mathcal{A}$-measurable functions and the $\mathcal{A}$-filters ($\mathcal{A}$-ultrafilters) on $X$. 
%\marginpar{I rephrased the following two sentences.}%
We have defined real and hyperreal maximal ideals in $\mathcal{M}(X,\mathcal{A})$ in \cite{ABS}.  
In this paper, we define realcompact measurable spaces and show that two separated realcompact measurable spaces $(X,\mathcal{A})$ and $(Y,\mathcal{B})$ are isomorphic if and only if the rings  $\mathcal{M}(X,\mathcal{A})$ and $\mathcal{M}(Y,\mathcal{B})$ are isomorphic (Theorem~\ref{thm:MeasureIsoRingIso}). This is the first principal technical result of the present paper obtained in section 2.  This may be called the measure theoretic analogue of the celebrated Hewitt's isomorphism theorem in the rings of continuous functions (see \cite[Theorem 8.3]{GJ}). In the same section we have further established that given a separated measurable space $(X,\mathcal{A})$, not necessarily realcompact, the ring $\mathcal{M}(X,\mathcal{A})$ can be embedded monomorphically inside the ring $C(R Max(X,\mathcal{A}))$ of all real valued continuous functions defined over the space $R Max(X,\mathcal{A})$. Here $R Max(X,\mathcal{A})$ represents the set of all real maximal ideals of $\mathcal{M}(X,\mathcal{A})$ with the topology inherited from the hull-kernel topology of the parent space $Max(X,\mathcal{A})$ of all maximal ideals of $\mathcal{M}(X,\mathcal{A})$ (Theorem~\ref{thm:separatedYieldsMonomorphism}).

In section 3 of this article, we introduce, via a measure $\mu$ on $(X,\mathcal{A})$, the $m_\mu$-topology on the ring $\mathcal{M}(X,\mathcal{A})$. We realize that the $m_\mu$-topology is connected if and only if it is 1st countable. Incidentally, any of these last two conditions is proved to be equivalent to the requirement that $\mathcal{M}(X,\mathcal{A})$ is identical to the subring $L^\infty(\mu)$ of all $\mu$-essentially bounded measurable functions on $(X,\mathcal{A})$ (Theorem~\ref{thm:mMuEquivalence}). A special case of this theorem on choosing $\mu$ to be the counting measure on $(X,\mathcal{A})$ reads: the $m$-topology on $\mathcal{M}(X,\mathcal{A})$ is 1st countable if and only if the $\sigma$-algebra $\mathcal{A}$ on $X$ is finite. This fact is proved in \cite[Theorem 2.16]{ABS}. 
In the same vein we also establish %by embedding the non reflexive Banach space $l^\infty$ of all bounded sequence of real numbers as a closed subset into the Banach space $\mathcal{M}^*(X,\mathcal{A})$ with $\mathcal{A}$ an infinite $\sigma$-algebra,%(let's skip the details here) 
that $\mathcal{M}^*(X,\mathcal{A})$, which is a Banach space with respect to the supremum norm, is reflexive if and only if $\mathcal{A}$ is finite (Theorem~\ref{thm:reflexive} and surrounding discussion). 

In the final and 4th section, %of this article% (unnecessary) 
we have investigated some problems related to two chosen subrings of $\mathcal{M}(X,\mathcal{A})$. The first one of these is the ring $\mathcal{M}_F(X,\mathcal{A})$ of all $\mathcal{A}$-measurable functions which vanish at all but finitely many points in $X$. The second one is the ring $M_\infty (X,\mathcal{A})$ of all such functions $f$ for which $\{x\in X:\lvert f(x)\rvert \geq \epsilon\}$ is at most a finite set for each $\epsilon >0$. It is relevant to mention that several aspects of these %rings and% 
rings have also been studied in \cite{AAMH2013}, \cite{AHM2009}, \cite{EMY2018}, \cite{G1966}, \cite{V1974}, \cite{V1977}, \cite{V1978}, \cite{V1981}.

\section{Measure theoretic version of Hewitt's isomorphism theorem} 
 Throughout this paper whenever we speak of an ideal unmodified we will always mean a proper ideal in the ring under consideration. We have already defined real and hyperreal maximal ideals in the ring $\mathcal{M}(X,\mathcal{A})$ in \cite{ABS}. In order to make the present paper self-contained, we reproduce these two definitions and a useful characterization of the realness of a maximal ideal.
 
\begin{definition}%\marginpar{Changed definition of real}%
A maximal ideal $M$ in $\mathcal{M}(X,\mathcal{A})$ is called real if the images under the canonical map $\pi : \mathcal{M}(X,\mathcal{A})\mapsto \mathcal{M}(X,\mathcal{A})/M$ given by $f\mapsto M(f)$ of just the constant functions is all of $\mathcal{M}(X,\mathcal{A})/M$, otherwise $M$ is called hyperreal. Here $M(f)$ stands for the residue class containing the function $f$. Whenever $M$ is a real maximal ideal of $\mathcal{M}(X,\mathcal{A})$ we say that $\mathcal{M}(X,\mathcal{A})/M$ is a real field.
\end{definition}
%\marginpar{Moved definition of separated to here from proof of Thm 2.2.}%
We say a $\sigma$-algebra $\mathcal{A}$ separates points if given two distinct points $p,q$ in $X$, there is a member $E$ of $\mathcal{A}$ which contains exactly one of them. 
We then say that the measurable space $(X,\mathcal{A})$ is separated. 
For any $f\in \mathcal{M}(X,\mathcal{A})$, we let $Z(f)=\{x\in X: f(x)=0\}$ stand for the zero set of $f$.
%\marginpar{Added the following sentence, and added 2.7 to reference.}
The following theorem is from \cite[Theorems 2.7 and 3.9]{ABS}.

\begin{theorem}
A maximal ideal $M$ in $\mathcal{M}(X,\mathcal{A})$ is real if and only if the $\mathcal{A}$-ultrafilter $Z[M]\equiv \{Z(f):f\in M\}$ of $\mathcal{A}$-measurable sets is closed under countable intersection %if and only if $Z[M]$ has countable intersection property% (repetition)
%[Theorem 3.9 in []]. 
\end{theorem}

An ideal $I$ in $\mathcal{M}(X,\mathcal{A})$ (respectively $\mathcal{M}^*(X,\mathcal{A})$) is called fixed if $\cap Z[I]:= \displaystyle \bigcap_{f\in I} Z(f)\neq \emptyset$, otherwise $I$ is called a free ideal. It is observed in \cite[Proposition 6]{EMY2018} that the complete list of fixed maximal ideals in $\mathcal{M}(X,\mathcal{A})$ is given by $\{M_p:p\in X \}$, where $M_p=\{f\in \mathcal{M}(X,\mathcal{A}):f(p)=0\}$. 
It is easy to see that if in addition the $\sigma$-algebra $\mathcal{A}$ on $X$ separates points, then $M_p\neq M_q$ whenever $p\neq q$ in $X$. Each fixed maximal ideal in $\mathcal{M}(X,\mathcal{A})$ is real, an easy verification. If conversely each real maximal ideal of $\mathcal{M}(X,\mathcal{A})$ is fixed, $(X,\mathcal{A})$ is called a realcompact space. 

If $Max (X,\mathcal{A})$ denotes the set of all maximal ideals of the ring $\mathcal{M}(X,\mathcal{A})$ and for $f\in \mathcal{M}(X,\mathcal{A}), \mathcal{B}_f=\{M\in Max(X,\mathcal{A}):f\in M\}$, then it is easy to establish by using some standard arguments that $\{\mathcal{B}_f:f\in \mathcal{M}(X,\mathcal{A})\}$ constitutes a base for the closed sets of $Max(X,\mathcal{A})$ if the hull-kernel topology on the later set is imposed. We will use the following two notations: $R Max(X,\mathcal{A})=\{M\in Max(X,\mathcal{A}): M$ is real maximal ideal$\}$, $F Max(X,\mathcal{A})=\{M_p:p\in X\}\equiv$ the set of all fixed maximal ideals of $\mathcal{M}(X,\mathcal{A})$. It is clear that $F Max(X,\mathcal{A})\subseteq R Max(X,\mathcal{A})$.

\begin{definition}
A bijective map $t: (X,\mathcal{A})\mapsto (Y,\mathcal{B})$ is called an isomorphism between these two measurable spaces if for $E\subseteq X, E\in \mathcal{A}$ if and only if $t(E)\in \mathcal{B}$. 
\end{definition}
	
It is easy to check that any such isomorphism $t$, induces an isomorphism $\psi_t: \mathcal{M}(Y,\mathcal{B})\mapsto \mathcal{M}(X,\mathcal{A})$ by $g\mapsto g\circ t$ between the two rings. We shall show that with the assumptions of realcompactness and separateness of the measure spaces $(X,\mathcal{A})$ and $(Y,\mathcal{B})$, any ring isomorphism between $\mathcal{M}(X,\mathcal{A})$ and $\mathcal{M}(Y,\mathcal{B})$ will generate an isomorphism between $(X,\mathcal{A})$ and $(Y,\mathcal{B})$. We need to establish several subsidiary results to prove this basic fact.

\begin{theorem}
	The family $\mathcal{B}^R=\{\mathcal{B}_f^R:f\in \mathcal{M}(X,\mathcal{A})\}$, where $\mathcal{B}_f^R=\mathcal{B}_f\cap R Max(X,\mathcal{A})$ is a $\sigma$-algebra over the set $R Max(X,\mathcal{A})$.
\end{theorem}

\begin{proof}
$\mathcal{B}_1^R=\emptyset$, since there are no proper ideals containing unity. Thus $\emptyset \in \mathcal{B}^R$. Next we show that for any $f\in \mathcal{M}(X,\mathcal{A}), (\mathcal{B}_f^R)^c\equiv R Max (X,\mathcal{A})\setminus \mathcal{B}_f^R=\mathcal{B}_g^R$, where $g=\Large\chi_{Z(f)}$, the characteristic function of $Z(f)$. Indeed if $f\in (\mathcal{B}_f^R)^c$, then $f\notin M$. Since $fg=0$ and the maximal ideal $M$ in $\mathcal{M}(X,\mathcal{A})$ is prime, this implies that $g\in M$ i.e. $M\in\mathcal{B}_g^R$. Conversely if $M\in\mathcal{B}_g^R$, then $g\in M$ which implies that $Z(g)\in Z[M]$. Since $Z[M]$ is an $\mathcal{A}$-ultrafilter on $X$ and $Z(f)\cap Z(g)=\emptyset$, this implies that $Z(f)\notin Z[M]$, because each $\mathcal{A}$-filter on $X$ has finite intersection property (see \cite[Definition 2.1]{ABS}). This surely implies that $f\notin M$ i.e. $M\in (\mathcal{B}_f^R)^c$.

Finally let $\{f_n \}_{n}$ be a countable family of functions in $\mathcal{M}(X,\mathcal{A})$. We show that $\displaystyle \bigcup_{n=1}^\infty \mathcal{B}_{f_n}^R=\mathcal{B}_g$ where $g$ is the characteristic function of the set $X\setminus \displaystyle \bigcup_{n=1}^\infty Z(f_n)$, clearly $g$ is measurable.

Proof of ``$\subseteq$": Let $M\in \mathcal{B}_{f_n}^R$ for some $n\in \mathbb{N}$, then $f_n\in M$. To show that $M\in B_g$ i.e. $g\in M$ it is sufficient to show that $Z(g)\supseteq Z(f_n)$, since by Theorem 2.3 of \cite{ABS}, this means that $g$ is a multiple of $f_n$ and hence in the same ideal $M$ as $f_n$. Indeed $Z(g)=\displaystyle \bigcup_{k=1}^\infty Z(f_k)\supseteq Z(f_n)$. Thus $\displaystyle \bigcup_{n=1}^\infty \mathcal{B}_{f_n}^R\subseteq \mathcal{B}_g^R$.

Proof of ``$\supseteq$": Suppose $M\in \mathcal{B}_g^R$, then $g\in M$, consequently $Z(g)=\displaystyle \bigcup_{n=1}^\infty Z(f_n)\in Z[M]$. Since $M$ is a real maximal ideal of $\mathcal{M}(X,\mathcal{A})$, it follows from Theorem 2.2 that $Z(f_n)\in Z[M]$ for some $n\in \mathbb{N}$. This implies that $f_n\in M$ as because each $\mathcal{A}$-ultrafilter on $X$ is maximal with respect to having finite intersection property (see \cite[Definition 2.1]{ABS}). Thus $M\in \mathcal{B}_{f_n}^R$.
\end{proof}

\begin{corollary}
	$\mathcal{B}^F\equiv\{\mathcal{B}_f^F: f\in F Max(X,\mathcal{A})\}$ is a $\sigma$-algebra over $F Max(X,\mathcal{A})$.
\end{corollary}

\begin{theorem}
	Any separated measurable space $(X,\mathcal{A})$ is isomorphic to the measurable space $(F Max(X,\mathcal{A}),\mathcal{B}^F)$.
\end{theorem}

\begin{proof}
The isomorphism is given by $\psi : X\mapsto F Max (X,\mathcal{A})$ mapping $p$ to $M_p$. It is injective because $\mathcal{A}$ is separated. Furthermore an arbitrary member of $\mathcal{A}$ is of the form $Z(f)$ for some $f\in \mathcal{M}(X,\mathcal{A})$, indeed $f=\Large\chi_{X\setminus A}$. We observe that $\psi (Z(f))=\{\psi (p): p\in Z(f)\}=\{M_p:f\in M_p\}=\mathcal{B}_f^F$.
\end{proof}

\begin{theorem}
If $t:\mathcal{M}(X,\mathcal{A})\mapsto \mathcal{M}(X,\mathcal{B})$ is a ring isomorphism, then a maximal ideal $M$ in $\mathcal{M}(X,\mathcal{A})$ is real if and only if $t(M)$ is a real maximal ideal of $\mathcal{M}(Y,\mathcal{B})$.
\end{theorem}

\begin{proof}
For any maximal ideal $M$ in $\mathcal{M}(X,\mathcal{A})$, the quotient $\mathcal{M}(X,\mathcal{A})/M$ is isomorphic to $\mathcal{M}(Y,\mathcal{B})/t(M)$ under the map $t:M(f)\mapsto t(M)(t(f))$. Hence it follows that $M$ is a real maximal ideal of $\mathcal{M}(X,\mathcal{A})$ if and only if $\mathcal{M}(X,\mathcal{A})/M$ is isomorphic to $\mathbb{R}$ if and only if $\mathcal{M}(Y,\mathcal{B})/t(M)$ is isomorphic to $\mathbb{R}$ meaning that $t(M)$ is a real maximal ideal of $\mathcal{M}(Y,\mathcal{B})$. Surely this assertion can equivalently be written in the following notations: for any $f\in\mathcal{M}(X,\mathcal{A}), t(\mathcal{B}_f^R)=\mathcal{B}_{t(f)}^R$.
\end{proof}	

\begin{remark}
The above ring isomorphism $t:\mathcal{M}(X,\mathcal{A})\mapsto \mathcal{M}(Y,\mathcal{B})$ induces an isomorphism $\hat{t}:(R Max (X,\mathcal{A}),\mathcal{B}^R)\mapsto (R Max (Y,\mathcal{B}),\mathcal{B}^R):\mathcal{B}_f^R\mapsto \mathcal{B}_{t(f)}^R$ between the respective measurable spaces of their real maximal ideals. We have organized our machinery and are ready to present the proof of the first major theorem of this section. 
\end{remark}

\begin{theorem}\label{thm:MeasureIsoRingIso}
Two separated realcompact measurable spaces $(X,\mathcal{A})$ and $(Y, \mathcal{B})$ are isomorphic if and only if the rings $\mathcal{M}(X,\mathcal{A})$ and $\mathcal{M}(Y,\mathcal{B})$ are isomorphic.
\end{theorem}

\begin{proof}
If $(X,\mathcal{A})$ and $(Y,\mathcal{B})$ are isomorphic then it is trivial and we have already checked [vide Theorem 2.3] that $\mathcal{M}(X,\mathcal{A})$ and $\mathcal{M}(Y,\mathcal{B})$ are isomorphic. Conversely let $\mathcal{M}(X,\mathcal{A})$ and $\mathcal{M}(Y,\mathcal{B})$ be isomorphic as rings. Then from Remark 2.8, it follows that the measurable spaces $R Max(X,\mathcal{A},\mathcal{B}^R)$ and $R Max(Y,\mathcal{B},\mathcal{B}^R)$ become isomorphic. Since $(X,\mathcal{A})$ is realcompact, we should have $R Max(X,\mathcal{A})=F Max(X,\mathcal{A})$ and for the same reason, $RMax (Y,\mathcal{B})=F Max (Y,\mathcal{B})$. Thus the two measurable spaces $(F Max (X,\mathcal{A}),\mathcal{B}^F)$ and $(F Max (Y,\mathcal{B}),\mathcal{B}^F)$ are isomorphic. Since $(X,\mathcal{A})$ and $(Y,\mathcal{B})$ are separated it follows from Theorem 2.6 that $(X,\mathcal{A})$ and $(Y,\mathcal{B})$ are isomorphic.
\end{proof}

\begin{remark}%\marginpar{What does this remark refer to?}
Since every compact measurable space $(X,\mathcal{A})$ is realcompact, the corollary 5 of \cite{EMY2018} is a special case of Theorem 2.9. We would like to mention that compact measurable spaces were introduced in \cite{EMY2018}. It was realized in \cite{ABS} that a measurable space $(X,\mathcal{A})$ is compact if and only if $\mathcal{A}$ is finite. 
\end{remark}

Next, We show that for a separated measurable space $(X,\mathcal{A})$, the ring $\mathcal{M}(X,\mathcal{A})$ is actually isomorphic to a subring of the form $C(Y)$, where $Y$ is a zero-dimensional Hausdorff topological space. In this connection we recall the fact that $Max(X,\mathcal{A})$ with hull-kernel topology is a compact Hausdorff zero-dimensional space, a result already established in \cite[Theorem 2.10]{ABS}. Since every subspace of a zero-dimensional space is zero-dimensional, it follows that the space $R Max(X,\mathcal{A})$ of all real maximal ideals of $\mathcal{M}(X,\mathcal{A})$ is a Hausdorff zero-dimensional space.

\begin{theorem}\label{thm:separatedYieldsMonomorphism}
Let $(X,\mathcal{A})$ be a separated measurable space. %\marginpar{Changed homomorphism to monomorphism}
Then there exists a ring monomorphism $\psi :\mathcal{M}(X,\mathcal{A})\mapsto C(R Max (X,\mathcal{A}))$.
\end{theorem}

\begin{proof}
Let $e: X\mapsto Max(X,\mathcal{A})$ be the map defined by $e(x)=M_x\equiv \{h\in \mathcal{M}(X,\mathcal{A}): h(p)=0\}$. We choose $f\in \mathcal{M}(X,\mathcal{A})$. We shall show that this ``$f$" induces a continuous map $f^\beta : R Max (X,\mathcal{A})\mapsto \mathbb{R}$ with the following property: $f^\beta \circ e=f$. For that purpose, select $M\in R Max(X,\mathcal{A})$ arbitrarily i.e. $M$ is a real maximal ideal of $\mathcal{M}(X,\mathcal{A})$. Set $\tilde{M}=\{g\in\mathcal{M}(\mathbb{R},\mathcal{B}(\mathbb{R})):g\circ f\in M\}$, here $\mathcal{B}(\mathbb{R})$ is the $\sigma$-algebra of all Borel sets in $\mathbb{R}$ and $(\mathbb{R},\mathcal{B}(\mathbb{R}))$ is the corresponding measurable space on $\mathbb{R}$. Since the compose of a Borel measurable function $g:\mathcal{B}(\mathbb{R}))\mapsto (\mathbb{R},\mathcal{B}(\mathbb{R}))$ with a real valued measurable function on $(X,\mathcal{A})$ is again a measurable function on $(X,\mathcal{A})$, an easy verification, and since also the maximal ideal $M$ in the ring $\mathcal{M}(X,\mathcal{A})$ is prime, it follows that $\tilde{M}$ is a prime ideal in the ring $\mathcal{M}(\mathbb{R},\mathcal{B}(\mathbb{R}))$. It is a standard result in the theory of rings of measurable functions that each prime ideal in any ring of the form $\mathcal{M}(Y,\mathcal{B})$ with $(Y,\mathcal{B})$ a measurable space is maximal [see \cite{ABS}, corollary 2.4]. Thus $\tilde{M}$ is a maximal ideal in the ring $\mathcal{M}(X,\mathcal{A})$. We next show that $\tilde{M}$ is a real maximal ideal of this ring. It suffices to check in view of Theorem 2.2 that the $\mathcal{B}(\mathbb{R})$-ultrfilter $Z[\tilde{M}]$ on $\mathbb{R}$ has countable intersection property. Toward that claim, choose a countable family of functions $\{g_n\}_{n=1}^\infty$ from $\tilde{M}$. It follows that for each $n\in\mathbb{N}, g_n\circ f\in M$. The hypothesis that $M$ is a real maximal ideal of $\mathcal{M}(X,\mathcal{A})$ implies in view of Theorem 2.2 that $\displaystyle \bigcap_{n=1}^\infty Z(g_n\circ f)\neq \emptyset$. There exists therefore $x\in X$ such that $g_nf(x)=0$ for each $n\in\mathbb{N}$. This yields that $f(x)\in \displaystyle \bigcap_{n=1}^\infty Z(g_n)$, thus $\displaystyle \bigcap_{n=1}^\infty Z(g_n)\neq \emptyset$. To show that $\tilde{M}$ is actually a fixed maximal ideal of $\mathcal{M}(\mathbb{R},\mathcal{B}(\mathbb{R}))$ we notice that the continuous function $i:\mathbb{R}\mapsto \mathbb{R}$ given by $i(r)=r$ is a member of $\mathcal{M}(\mathbb{R},\mathcal{B}(\mathbb{R}))$. since $\tilde{M}$ is real maximal ideal, there exists $r\in \mathbb{R}$ such that $\tilde{M}(i)=\tilde{M}(\underline{r})$ in the residue class field $\mathcal{M}(\mathbb{R},\mathcal{B}(\mathbb{R}))/\tilde{M}$, here $\underline{r}$ is the constant function on $\mathbb{R}$ with value $r$. this yields that $i-\underline{r}\in\tilde{M}$ and hence $Z(i-\underline{r})\in Z[\tilde{M}]$. As the function $i-\underline{r}$ on $\mathbb{R}$ can vanish at most at one point and hence exactly one point, it follows that $Z(i-\underline{r})$ is a one-pointic set. Consequently $\tilde{M}$ is fixed maximal ideal of $\mathcal{M}(\mathbb{R},\mathcal{B}(\mathbb{R}))$ i.e. $\cap Z[\tilde{M}]$ is a one-point set.

We now set $f^\beta(M)=\cap Z[\tilde{M}]$. this defines a function $f^\beta : R Max(X,\mathcal{A})\mapsto \mathbb{R}$. If $g\in \mathcal{M}(\mathbb{R}, \mathcal{B}(\mathbb{R}))$ and $x\in X$ are such that $g\circ f\in M_x$, then it follows that $f(x)\in Z(g)$. Consequently $f(x)\in f^\beta (M_x)$. But $f^\beta (M_x)=(f^\beta \circ e)(x)$. Hence $f(x)=(f^\beta \circ e)(x)$. thus $f^\beta \circ e=f$.

To check the continuity of $f^\beta$ at a point $M$ on $R Max(X,\mathcal{A})$, let $U$ be an open set in the space $\mathbb{R}$ with $f^\beta (M)\in U$. Using the complete regularity of the space $\mathbb{R}$, we can find out $g,h\in C(\mathbb{R})$ such that $f^\beta (M)\in \mathbb{R}\setminus Z(g)\subseteq Z(h)\subseteq U$. We notice that $gh=0$. As $f^\beta \notin Z(g)$, it follows that $g\notin \tilde{M}$, hence $g\circ f\notin M$, consequently, $M\notin \mathcal{B}_{g\circ f}^R$, the set of all real maximal ideals of $\mathcal{M}(X,\mathcal{A})$ which contains the function $g\circ f$. This further implies that $M\in R Max(X,\mathcal{A})\setminus \mathcal{B}_{g\circ f}^R$, which is an open neighborhood of $M$ in the space $R Max(X,\mathcal{A})$. We now assert that $f^\beta (R Max (X,\mathcal{A})\setminus \mathcal{B}_{g\circ f}^R)\subseteq U$ and that settles the continuity of $f^\beta$ at the point $M$. Toward the proof of the last assertion, let $N\in R Max(X,\mathcal{A})\setminus \mathcal{B}_{g\circ f}^R$. Then $g\circ f\notin N$ and hence $g\notin \tilde{N}$. Since $gh=0$, as noticed earlier, it follows because of the primeness of the ideal $\tilde{N}$ that $h\in \tilde{N}$. This implies that $f^\beta (N)\in Z(h)\subseteq U$, thus it is proved that $f^\beta : R Max(X,\mathcal{A})\mapsto \mathbb{R}$ is a continuous map i.e. $f^\beta \in C(R Max(X,\mathcal{A}))$. Now if $x\in X$ and $f,g\in \mathcal{M}(X,\mathcal{A})$, then we observe that $(f+g)^\beta(e(x))=[(f+g)^\beta\circ e](x)=(f+g)(x)=f(x)+g(x)=(f^\beta\circ e)(x)+(g^\beta \circ e)(x)=(f^\beta+g^\beta)\circ e(x)$. These last relations show that the two functions $(f+g)^\beta$ and $f^\beta+g^\beta$ defined on $R Max(X,\mathcal{A})$ to $\mathbb{R}$ agree on $e(X)$, which is dense $Max(X,\mathcal{A})$ and hence dense in $R Max(X,\mathcal{A})$. The denseness of $e(X)$ in the space $Max(X,\mathcal{A})$ follows from the definition of hull-kernel topology taking care of the fact that $\mathcal{M}(X,\mathcal{A})$ is a reduced ring meaning that $'0'$ is the only nilpotent element of it. Thus we can say that $(f+g)^\beta$ and $f^\beta+g^\beta$ agree on $cl_{RMax(X,\mathcal{A})}e(X)=RMax(X,\mathcal{A})$. Analogously $(fg)^\beta=f^\beta g^\beta$. therefore the map $\psi: \mathcal{M}(X,\mathcal{A})\mapsto C(R Max(X,\mathcal{A}))$ by $f\mapsto f^\beta$ is a ring homomorphism which is clearly one-to-one because $f\neq g$ in $\mathcal{M}(X,\mathcal{A})$ implies that $f(x)\neq g(x)$ for some $x\in X$. Consequently $f^\beta(e(x))=f(x)\neq g(x)=g^\beta(e(x))$.
\end{proof}

\begin{question}
	When can we write $\psi(\mathcal{M}(X,\mathcal{A}))=C(R Max(X,\mathcal{A}))?$
\end{question}

We shall construct an example of a separated measurable space $(X,\mathcal{A})$ for which the zero-dimensional Hausdorff space $RMax(X,\mathcal{A})$ is devoid of any isolated point. Before that let us examine the possible candidates for the isolated points in this space. The following proposition gives us the necessary information in this regard.

\begin{theorem}
Let $M$ be a real maximal ideal of $\mathcal{M}(X,\mathcal{A})$. If $M$ is a free maximal ideal, then it is never an isolated point of the space $RMax(X,\mathcal{A})$. On the other hand if $M$ is a fixed maximal ideal say $M=M_x, x\in X$, then $M$ is an isolated point of $RMax(X,\mathcal{A})$ if and only if $\{x\}\in \mathcal{A}$.
\end{theorem}

\begin{proof}
If possible let a free maximal ideal $M$ be an isolated point of $RMax(X,\mathcal{A})$. Then there exists an $f\in \mathcal{M}(X,\mathcal{A})$ such that $\{M\}=RMax(X,\mathcal{A})\setminus\mathcal{B}_f^R$. this implies that each maximal ideal $N$ in $\mathcal{M}(X,\mathcal{A})$ other than $M$ contains the function $f$ as a member of it. In particular therefore $f$ belongs to each fixed maximal ideal of $\mathcal{M}(X,\mathcal{A})$ and hence $f=0$. Consequently $\mathcal{B}_f^R=RMax(X,\mathcal{A})$ and hence $\{M\}=\phi$, a contradiction.

To prove the second part of the theorem, let $x\in X$ be such that $\{x\}\in \mathcal{A}$. To show that $M_x$ is isolated in $RMax(X,\mathcal{A})$, it suffices to show on setting $f=\chi_{\{x\}}$ that $\{M_x\}=RMax(X,\mathcal{A})\setminus\mathcal{B}_f^R$. It is clear that $f(x)\neq 0$ implies that $f\notin M_x$ and hence $M_x\notin \mathcal{B}_f^R$. Let $M$ be a maximal ideal in $\mathcal{M}(X,\mathcal{A})$ other than $M_x$. It is sufficient to show that $f\in M$, equivalently $M\in \mathcal{B}_f^R$. If $M$ is a fixed maximal ideal, then $M=M_y$ for some $y\neq x$ in $X$. since $f(y)=0$, this implies that $f\in M_y$. Assume therefore that $M$ is a free maximal ideal of $\mathcal{M}(X,\mathcal{A})$. Then there exists $g\in M$ with $g(x)\neq 0$. It follows that $Z(f)\supseteq Z(g)$. Consequently by Theorem 2.3 of \cite{ABS}, $f$ becomes a multiplier of $g$. Since $g\in M$, it follows that $f\in M$.

Conversely let $M_x$ be an isolated point of $RMax(X,\mathcal{A})$. Then there exists $f\in \mathcal{M}(X,\mathcal{A})$ such that $\{M_x\}=RMax(X,\mathcal{A})\setminus\mathcal{B}_f^R$. Then $f\notin M_x$, implying that $f(x)\neq 0$. Also for any $y\neq x$ in $X$, $f\in M_y$, which implies $f(y)=0$. Therefore cozero set of $f=X\setminus Z(f)=\{x\}$ is a member of the $\sigma$-algebra $\mathcal{A}$. 
\end{proof}	

\begin{example}[\textbf{An $RMax(X,\mathcal{A})$ space without any isolated point}] In view of Theorem 2.11, it suffices to construct a $\sigma$-algebra $\mathcal{A}$ on a suitable uncountable set $X$, for which the measurable space $(X,\mathcal{A})$ is separated but for each $x\in X, \{x\}\notin \mathcal{A}$. Let $X=[0,1]^{\omega_1}$, here $\omega_1$ is the 1st uncountable ordinal number. For each $\alpha<{\omega_1}$, let $f_\alpha: [0,1]^{\omega_1}\mapsto [0,1]^\alpha$ be the $\alpha$ th projection map defined in the usual manner. Let $\mathcal{A}=\{E\subseteq X:$ there exists $\alpha<{\omega_1}$ and $Y\subseteq [0,1]^\alpha$ such that $E=f_\alpha^{-1}(y)\}$. Since the least upper bound of any countable set of countable ordinals is again a countable ordinal, it follows that $\mathcal{A}$ is closed under countable union. It is easy to check that $\mathcal{A}$ is a $\sigma$-algebra over $X$ and each nonempty member of $\mathcal{A}$ is an uncountable set. It follows that for each $x\in X, \{x\}\notin \mathcal{A}$. The proof that $\mathcal{A}$ separates points of $X$ is routine, indeed if $x$ and $y$ are distinct points of $X$, then they must differ at their $\alpha$ th co-ordinate for some $\alpha<\omega_1$. Let $x^*=f_\alpha(x)$ and $y^*=f_\alpha(y)$. Then the set $f_\alpha^{-1}([0,1]^\alpha\setminus\{x^*\})$ is a member of $\mathcal{A}$, contains the point $y$ but misses the point $x$.
\end{example}

\section{The $m_\mu$-topology on the ring $\mathcal{M}(X,\mathcal{A})$}

In this section, we define a topology on $\mathcal{M}(X,\mathcal{A})$, namely the $m_\mu$-topology, through a measure $\mu$ on the $\sigma$-algebra $\mathcal{A}$. In the process, a new intermediate ring namely $L^\infty(\mu)$ consisting of all $\mu$-essentially bounded measurable functions in $\mathcal{M}(X,\mathcal{A})$ comes into the surface. The $m_\mu$-topology on $\mathcal{M}(X,\mathcal{A})$ which clearly depends on the measure $\mu$ turns out to be a measure theoretic generalization of the so-called $m$-topology on $\mathcal{M}(X,\mathcal{A})$. This later topology was introduced in \cite[Definition 2.14 and Theorems 2.15 and 2.16]{ABS} . We recall that a measure $\mu$ on $(X,\mathcal{A})$ is a nonnegative set function defined for all members of $\mathcal{A}$ with the two properties: $\mu (\emptyset)=0$ and $
\mu \left(\displaystyle \bigcup ^{\infty}_{n=1} E_{n}\right) = \displaystyle \sum^{\infty}_{n=1} \mu(E_{n})$, for a sequence $\{E_n\}_n$ of disjoint members of $\mathcal{A}$. For each $g$ in $\mathcal{M}(X,\mathcal{A})$ and each positive unit $u$ of this ring, set $m_\mu(g,u)=\{f\in\mathcal{M}(X,\mathcal{A}):\lvert f(x)-g(x)\rvert < u(x)$, for all $x$ almost everywhere with respect to the measure $\mu$ on $X\}$. It is routine to check that there exists a unique topology which we call $m_\mu$-topology on $\mathcal{M}(X,\mathcal{A})$ in which for each $g\in\mathcal{M}(X,\mathcal{A}), \{m_\mu(g,u):u$ is a positive unit of $\mathcal{M}(X,\mathcal{A})\}$ is a base for its open neighbourhoods. It is not at all hard to prove by using some routine arguments and the fact that a continuous function of a real valued measurable function is measurable, that $\mathcal{M}(X,\mathcal{A})$ with this $m_\mu$-topology is a topological ring. If $\mu$ is the counting measure on $(X,\mathcal{A})$, then $m_\mu$-topology reduces to the $m$-topology on $\mathcal{M}(X,\mathcal{A})$, already introduced in \cite{ABS}. Unlike $m$-topology, the set of all multiplicative units of the ring $\mathcal{M}(X,\mathcal{A})$ is not necessarily an open set in the $m_\mu$-topology for an arbitrary measure $\mu$. 

\begin{definition}
$f\in \mathcal{M}(X,\mathcal{A})$ is called a $\mu$-unit in this ring if $\mu (Z(f))=0$.
\end{definition}
%Because of the independent interest, we include a proof of the following result.

\begin{theorem}
Let $U_\mu$ be the set of all $\mu$-units in the ring $\mathcal{M}(X,\mathcal{A})$. Then $U_\mu$ is open in the $m_\mu$-topology.
\end{theorem}

\begin{proof}
Let $u\in U_\mu$. Let $Z(u)=M$. It is easy to verify that $\frac{\lvert u\rvert}{2}+\chi_{M}$ is a positive unit in the ring $\mathcal{M}(X,\mathcal{A})$. We claim that $m_\mu (u,\frac{\lvert u\rvert}{2}+\chi_{M})\subseteq U_\mu$, which will complete the proof. Indeed for any $v\in m_\mu(u,\frac{\lvert u\rvert}{2}+\chi_M)$, we have $\mu (N)=0$, where $N=\{x\in X:\lvert v(x)-u(x)\rvert\geq \frac{\lvert u(x)\rvert}{2}+\chi_M(x)\}$. Also since $u\in U_\mu, \mu(M)=0$. Clearly for any $x\in (X\setminus N)\cap (X\setminus M), v(x)\neq 0$. Therefore $(X\setminus N)\cap (X\setminus M)\subseteq X\setminus Z(v)$; which clearly implies that $Z(v)\subseteq M\cup N$. Since $\mu (M)=\mu (N) = 0$, it follows by monotonicity and countable sub-additivity of $\mu$ that $\mu(Z(v))=0$, which proves our claim.
\end{proof}

\begin{corollary}
A maximal ideal in $\mathcal{M}(X,\mathcal{A})$ containing no $\mu$-units is closed in $m_\mu$-topology.	
\end{corollary}

\begin{remark}
On choosing $\mu=$ counting measure on $(X,\mathcal{A})$, the above theorem reads: each maximal ideal is closed in $\mathcal{M}(X,\mathcal{A})$ in the $m$-topology. This fact is observed in \cite{ABS}.
\end{remark}

\begin{definition}
Let $L^\infty (\mu)=\{f\in \mathcal{M}(X,\mathcal{A}):\|f\|_\infty <\infty \}$, where $\|f\|_\infty =inf\{M:\mu\{t\in X:\lvert f(t)\rvert>M\}=0\}$
\end{definition}

It is routine to check that $L^\infty(\mu)$ is a subring of $\mathcal{M}(X,\mathcal{A})$ containing $\mathcal{M}^*(X,\mathcal{A})$. Furthermore it is not difficult to verify that $(L^\infty(\mu),+,.,\|\|_\infty)$ is a topological ring. The corresponding topology is generally known as the essential supnorm topology. Clearly the essential supnorm topology on $L^\infty(\mu)$ is weaker than the topology inherited from the $m_\mu$-topology on the parent space $\mathcal{M}(X,\mathcal{A})$. A natural question is: when do these two topologies coincide? The following proposition gives several conditions involving the $m_\mu$-topology, each equivalent to the above requirement.

\begin{theorem}\label{thm:mMuEquivalence}
The following statements are equivalent.
\begin{itemize}
\item[(a)] $L^\infty(\mu)=\mathcal{M}(X,\mathcal{A})$.
\item[(b)] The essential supnorm topology on $L^\infty(\mu)$ is identical to the relative $m_\mu$-topology on it.
\item[(c)] The $m_\mu$-topology on $\mathcal{M}(X,\mathcal{A})$ is connected.
\item[(d)] The $m_\mu$-topology on $\mathcal{M}(X,\mathcal{A})$ is 1st countable.
\end{itemize}
\end{theorem}

\begin{proof}
$(a)\Rightarrow (b):$ Assume that $\mathcal{M}(X,\mathcal{A})=L^\infty(\mu)$. Consider the set $m_\mu(g,u)$ where $g$ and $u$ are both arbitrary. Since $u$ is a unit in $L^\infty(\mu)$, it is easy to check that there exists $\lambda >0$ such that $u(x)\geq \lambda$ almost every where on $X$. Clearly this would imply $g\in \{f\in L^\infty(\mu):\|f-g\|_\infty\leq \lambda\}\subseteq m_\mu(g,u)$, which shows that $m_\mu(g,u)$ is a neighbourhood of $g$ in $L^\infty(\mu)$ in the essential supnorm topology. This is sufficient to conclude the two topologies coincide on $L^\infty(\mu)$. 

$(b)\Rightarrow (a):$ Suppose $L^\infty(\mu)\subsetneq \mathcal{M}(X,\mathcal{A})$. Then there exists $f\in \mathcal{M}(X,\mathcal{A})\setminus L^\infty(\mu)$ such that $f\geq 1$. Since $f\notin L^\infty(\mu), \mu \{t\in X:f(t)>n\}>0$ for all $n\in\mathbb{N}$. Let $g=\frac{1}{f}$. Then $g$ is a positive unit of $\mathcal{M}(X,\mathcal{A}), g\leq 1$ and $\mu \{t\in X: g(t)<\frac{1}{n}\}>0$ for all $n\in\mathbb{N}$. Thus $g$ takes arbitrarily small values near zero on sets of positive measure. We now consider $m_\mu(0,g)$. Clearly $m_\mu(0,g)\subseteq L^\infty(\mu)$.

Therefore $m_\mu(0,g)$ is a neighbourhood of 0 in $L^\infty(\mu)$ in the relative $m_\mu$-topology. We claim that $m_\mu(0,g)$ is not a neighbourhood of $0$ in $L^\infty(\mu)$ in the essential supnorm topology and this will imply that the statement $(b)$ is false. Indeed for every $\epsilon>0, \frac{\epsilon}{2}>g(x)$ on a set of positive measure which implies $\frac{\epsilon}{2}\notin m_\mu(0,g)$. Clearly $\frac{\epsilon}{2}\in\{f\in L^\infty(\mu):\|f\|_\infty \leq \epsilon\}$. This shows that $\{f\in L^\infty(\mu):\|f\|_\infty \leq \epsilon\}\nsubseteq m_\mu(0,g),$ which proves our claim.

Thus the statements $(a)$ and $(b)$ are equivalent.

Before establishing the equivalence of the statements $(c)$ and $(d)$ with $(a)$ or $(b)$, we make a simple observation that $L^\infty(\mu)$ is a clopen subset of $\mathcal{M}(X,\mathcal{A})$ in the $m_{\mu}$-topology. If $\mathcal{M}(X,\mathcal{A})=L^\infty(\mu)$, then then this claim is trivially true. Suppose $\mathcal{M}(X,\mathcal{A})\neq L^\infty(\mu)$. Let $f\in \mathcal{M}(X,\mathcal{A})\setminus L^\infty(\mu)$. Then it is clear that $m_\mu (f,1)\subseteq \mathcal{M}(X,\mathcal{A})\setminus L^\infty(\mu)$, which proves that $L^\infty(\mu)$ is a closed subset of $\mathcal{M}(X,\mathcal{A})$. Also it is easy to check for any $g\in L^\infty(\mu)$, that $m_\mu(g,1)\subseteq L^\infty(\mu)$ which shows that $L^\infty(\mu)$ is open in $\mathcal{M}(X,\mathcal{A})$ in the $m_\mu$-topology.

Now like any pseudonormed linear space, the space $L^\infty(\mu)$ with essential supnorm topology is path connected, in particular connected. This implies that $(a)\Rightarrow (c)$ is true. On the other hand, if $(a)$ is false, then from the above arguments we see that $L^\infty(\mu)$ turns out to be a nonempty proper clopen set of $\mathcal{M}(X,\mathcal{A})$. Hence $\mathcal{M}(X,\mathcal{A})$ becomes disconnected in the $m_\mu$-topology. Thus $(a)\Leftrightarrow (c)$ is established.

Since every pseudo normed linear space is first countable, $(a)\Rightarrow (d)$ is trivial.

$(d)\Rightarrow (a):$ Suppose $(a)$ is false. We can choose $f\geq 0$ in $\mathcal{M}(X,\mathcal{A})$ such that $\|f\|_\infty=inf \{M:\mu \{t\in X:f(t)>M\}=0\}=\infty$. Therefore for each $n\in \mathbb{N}, \mu \{t\in X:f(t)>n\}>0$, in particular $\mu\{t\in X:f(t)>1\}>0$. Hence there exists $k_1\in \mathbb{N}, k_1>1$ such that $\mu \{t\in X:1<f(t)<k_1\}>0$. For if no such $k_1\in \mathbb{N}$ exists then due to the fact that $\mu \{t\in X: f(t)>1 \}=\mu (\displaystyle \bigcup_{n=2}^\infty\{t\in X:1<f(t)<n\})$, it would imply that $\mu \{t\in X:f(t)>1 \}\leq \displaystyle \sum_{n=2}^\infty \mu \{t \in X:1<f(t)<n\}=0$, a contradiction. If we repeat this argument and take recourse to the principle of Mathematical induction, we get a strictly increasing sequence $1<k_1<k_2<....$ of natural numbers for which for each $i\in \mathbb{N}, \mu \{t\in X:k_i<f(t)<k_{i+1}\}>0$. Let for $i\in \mathbb{N}, E_i=\{t\in X:k_i<f(t)<k_{i+1}\}$. Then $E_i\neq \emptyset$ as $\mu (E_i)>0$. Furthermore $E_i\cap E_j=\emptyset$ if $i\neq j$. Thus $\{E_i\}_{i}$ is a countably infinite family of pairwise disjoint nonempty members of the $\sigma$-algebra $\mathcal{A}$ on $X$. We now assert that the $m_\mu$-topology on $\mathcal{M}(X,\mathcal{A})$ is not first countable at the point $0$ and that finishes the proof. 

Suppose toward a contradiction, let $\{m_\mu (0,u_n):n\in \mathbb{N}\}$ be a countable open base about the point $0$. Define $u:X\mapsto \mathbb{R}$ as follows: $u(x)=\frac{1}{2}u_n(x)$, if $x\in E_n, n\in \mathbb{N}$ and $u(x)=1$ if $x\in X\setminus \displaystyle \bigcup_{n=1}^\infty E_n$. Clearly $u$ is a measurable function and is a positive unit of the ring $\mathcal{M}(X,\mathcal{A})$. We notice that for each $n\in \mathbb{N}, \frac{2}{3}u_n\in m_\mu (0,u_n)\setminus m_\mu(0,u),$ because $\frac{2}{3}u_n$ is greater than $u$ on the set $E_n$ that has a positive measure. Thus $m_\mu(0,u_n)\subsetneq m_\mu (0,u)$ for each $n\in \mathbb{N}$, a contradiction to the first countability of the $m_\mu$-topology at the point $0$.
\end{proof}

\begin{remark}
A special version of the equivalence of $(a)$ and $(d)$ with $\mu\equiv$ counting measure reads: 
	
the $m$-topology on $\mathcal{M}(X,\mathcal{A})$ is 1st countable if and only if $\mathcal{A}$ is finite. This is established in \cite{ABS}.
	
We conclude this section by showing that the Banach space $\mathcal{M}^*(X,\mathcal{A})$ with supnorm topology is never reflexive unless the $\sigma$-algebra $\mathcal{A}$ on $X$ is finite. It is to be noted trivially that if $\mathcal{A}$ is finite $\sigma$-algebra containing `$n$' many members then $\mathcal{M}^*(X,\mathcal{A})$ is essentially the same as $\mathbb{R}^n$ as a Banach space and is therefore reflexive. 
\end{remark}

\begin{theorem}\label{thm:reflexive} %\marginpar{Should this thm be strengthened to reflect state $\mathcal{M}^*(X,\mathcal{A})$ contains a copy of $\ell^\infty$?  And then state this result as a corollary?}
Let the $\sigma$-algebra $\mathcal{A}$ on $X$ be infinite. Then $\mathcal{M}^*(X,\mathcal{A})$ contains a copy of $\ell^\infty$ as a closed subspace.
%is not reflexive.
\end{theorem}

\begin{proof}
Since $\mathcal{A}$ is infinite, there exists a pairwise disjoint sequence $\{E_n\}_{n}$ of nonempty members of $\mathcal{A}$ (vide \cite[Lemma 2.12]{ABS}). For each $i\in \mathbb{N}$, define $f_i$ in $\mathcal{M}(X,\mathcal{A})$ as follows.

\[   
f_i(E_j) = 
\begin{cases}
1, &\text{if } i=j\\
0 &\text{if } i\neq j \\
\end{cases}
\]
and $f_i=0$ on $X\setminus \displaystyle \bigcup_{j=1}^\infty E_j$.
Now define the function $\varphi :\ell^\infty\mapsto \mathcal{M}^*(X,\mathcal{A})$ by the formula: $\varphi \{a_n\}_n=\displaystyle \sum_{n=1}^\infty a_nf_n=f$, say, here $\ell^\infty$ is the Banach space of all bounded sequences of real numbers with supremum norm. It is clear that $\|\varphi \{a_n\}_n\|=\|\{a_n \}_n\|$, therefore $\varphi$ defines an isometric isomorphism on $l^\infty$ onto a closed subspace of $\mathcal{M}^*(X,\mathcal{A})$. 
\end{proof}
Since every closed subspace of a reflexive Banach space is reflexive and $l^\infty$ is a well known non-reflexive Banach space, it follows from Theorem~\ref{thm:reflexive} that when $X$ is infinite, then $\mathcal{M}(X,\mathcal{A})$ is not reflexive.

\section{Two special subrings of $\mathcal{M}(X,\mathcal{A})$}
We introduce the following two subrings of $\mathcal{M}^*(X,\mathcal{A})$.

\begin{definition}
$\mathcal{M}_F(X,\mathcal{A})=\{f\in\mathcal{M}(X,\mathcal{A}):X\setminus Z(f)$ is at most a finite set. 
\end{definition}

\begin{definition}
$\mathcal{M}_\infty(X,\mathcal{A})=\{f\in\mathcal{M}(X,\mathcal{A}):$ for all $\epsilon >0, \{x\in X:\lvert f(x)\rvert\geq\epsilon \}$ is at most finite set$\}$.
\end{definition}
It is easy to check that $\mathcal{M}_F(X,\mathcal{A})\subseteq \mathcal{M}_\infty(X,\mathcal{A})\subseteq\mathcal{M}^*(X,\mathcal{A})$. Furthermore, on using same routine arguments, it is not hard to verify that $\mathcal{M}_F(X,\mathcal{A})$ is an ideal of $\mathcal{M}(X,\mathcal{A})$ and hence of $\mathcal{M}^*(X,\mathcal{A})$ and $\mathcal{M}_\infty(X,\mathcal{A})$ is an ideal of $\mathcal{M}^*(X,\mathcal{A})$. $\mathcal{M}_\infty(X,\mathcal{A})$ is only a subring of $\mathcal{M}(X,\mathcal{A})$ but not necessarily an ideal of $\mathcal{M}(X,\mathcal{A})$. Indeed if $X=\mathbb{R}$ and $\mathcal{A}$ is the $\sigma$-algebra of all Borel sets in $\mathbb{R}$, then the function $f:X\mapsto\mathbb{R}$ given by $f(x)=\frac{1}{x}$ if $x\in\mathbb{N}$ and $f(x)=0$ otherwise is a function in $\mathcal{M}_\infty(X,\mathcal{A})$. We observe that the function $g:X\mapsto\mathbb{R}$ defined by $g(x)=x$ if $x\in\mathbb{N}$ and $g(x)=0$, otherwise belongs to $\mathcal{M}(X,\mathcal{A})$ but $fg\notin\mathcal{M}_\infty(X,\mathcal{A})$. We shall determine conditions both necessary and sufficient for $\mathcal{M}_\infty(X,\mathcal{A})$ to become an ideal of $\mathcal{M}(X,\mathcal{A})$ and also for the equality $\mathcal{M}_F(X,\mathcal{A})=\mathcal{M}_\infty(X,\mathcal{A})$.

It is clear that if the $\sigma$-algebra $\mathcal{A}$ on $X$ is finite then $\mathcal{M}_F(X,\mathcal{A})=\mathcal{M}^*(X,\mathcal{A})$. On the other hand if $\mathcal{A}$ is infinite then by Lemma 2.12 of \cite{ABS}, there exists an  infinite family $\{E_n\}_{n}$ of pairwise disjoint nonempty sets in $\mathcal{A}$. The function $f:X\mapsto \mathbb{R}$ defined by $f(E_n)=\frac{1}{n}, n\in\mathbb{N}$ and $f\left(X\setminus \displaystyle \bigcup_{n=1}^\infty E_n\right)=0$ belongs to $\mathcal{M}^*(X,\mathcal{A})\setminus\mathcal{M}_F(X,\mathcal{A})$. Thus $\mathcal{M}_F(X,\mathcal{A})$ is a proper ideal of $\mathcal{M}^*(X,\mathcal{A})$ if and only if $\mathcal{A}$ is infinite.

From now in this section we shall assume that $\mathcal{A}$ is an infinite $\sigma$-algebra on $X$ and the measurable space $(X,\mathcal{A})$ is separated. The next result connects $\mathcal{M}_F(X,\mathcal{A})$ with the free ideals of the rings $\mathcal{M}(X,\mathcal{A})$ and $\mathcal{M}^*(X,\mathcal{A})$. We recall from \cite{ABS}, that an ideal $I$(proper) in $\mathcal{M}(X,\mathcal{A})$ (respectively $\mathcal{M}^*(X,\mathcal{A})$) is called free if $\displaystyle \bigcap_{f\in I}Z(f)=\emptyset$.

\begin{theorem}
\begin{itemize}
\item[(a)] $\mathcal{M}_F(X,\mathcal{A})$ is equal to the intersection of all free ideals of $\mathcal{M}(X,\mathcal{A})$(respectively $\mathcal{M}^*(X,\mathcal{A})$).
\item[(b)] $\mathcal{M}_F(X,\mathcal{A})$ is a free ideal of $\mathcal{M}(X,\mathcal{A})$ (respectively $\mathcal{M}^*(X,\mathcal{A})$) if and only if $\mathcal{A}$ is locally finite in the sense that given $x\in X$, there is an $A\in\mathcal{A}$ such that $x\in A$ and $A$ is a finite set.[compare with 4D 3.5\cite{GJ}]. 
\end{itemize}
\end{theorem}

\begin{proof}
(a) Let $I$ be a free ideal of $\mathcal{M}(X,\mathcal{A})$. We show that $\mathcal{M}_F(X,\mathcal{A})\subseteq I$. Let $f\in \mathcal{M}_F(X,\mathcal{A})$. Then $X\setminus Z(f)$ is a finite set say $X\setminus Z(f)=\{x_1,x_2,...,x_n\}$. Since $I$ is free for each $i$, there is an $f_i\in I$ such that $f_i(x_i)\neq 0$. Take $g=f_1^2+f_2^2+...+f_n^2$. Then $g\in I$ and $Z(g)\cap (X\setminus Z(f))=\emptyset$. Thus $Z(g)\subseteq Z(f)$. Hence by pasting lemma \cite[Theorem 2.2]{ABS}, $f$ is a multiple of $g$ and therefore $f\in I$. Thus $\mathcal{M}_F(X,\mathcal{A})\subseteq I$. Conversely let $h\in \mathcal{M}(X,\mathcal{A})\setminus\mathcal{M}_F(X,\mathcal{A})$. Then there exists a countably infinite subset $\{a_1,a_2,....,a_n,...\}$ of $X\setminus Z(h)$. Let $J=\{k\in\mathcal{M}(X,\mathcal{A}):k$ vanishes at all but possibly finitely many $a_k$'s$\}$. Then $J$ is an ideal of $\mathcal{M}(X,\mathcal{A})$ and $h\notin J$. Also $\mathcal{M}_F(X,\mathcal{A})\subseteq J$. We now show that $J$ is a free ideal of $\mathcal{M}(X,\mathcal{A})$ and this completes the proof that $\mathcal{M}_F(X,\mathcal{A})$ is equal to the intersection of all free ideals of $\mathcal{M}(X,\mathcal{A})$. For that purpose, choose $x\in X$. Without loss of generality we assume that $x\neq a_n$ for each $n\in\mathbb{N}$. Then since $\mathcal{A}$ is separating, for each $n\in\mathbb{N}$, there is an $A_n\in\mathcal{A}$ such that $x\in A_n$ but $a_n\notin A_n$. Let $A=\displaystyle \bigcap_{n=1}^\infty A_n$. Then $x\in A\in \mathcal{A}$ and $\mathcal{A}\cap\{a_1,a_2,...,a_n,...\}=\emptyset$. It is then clear that the characteristic function $\chi_{X\setminus A}\in J$ and it does not vanish at the point $x$. Thus $J$ becomes a free ideal of $\mathcal{M}(X,\mathcal{A})$.

(b) We prove only the case when $\mathcal{M}_F(X,\mathcal{A})$ is a free ideal of $\mathcal{M}(X,\mathcal{A})$. The proof for the case when $\mathcal{M}_F(X,\mathcal{A})$ is also a free ideal of $\mathcal{M}^*(X,\mathcal{A})$ can be done analogously.

Let $\mathcal{M}_F(X,\mathcal{A})$ be a free ideal of $\mathcal{M}(X,\mathcal{A})$. Let $x\in X$. Then there is some $f\in \mathcal{M}_F(X,\mathcal{A})$ such that $f(x)\neq 0$. Hence $x\in X\setminus Z(f)$, which is a finite set in $\mathcal{A}$.

Conversely suppose $\mathcal{A}$ is locally finite. To show that $\mathcal{M}_F(X,\mathcal{A})$ is a free ideal, let $x\in X$. Then there exists $A\in\mathcal{A}$ such that $x\in A$ and $A$ is finite. It follows that the characteristic function $\chi_A$ is a member of $\mathcal{M}_F(X,\mathcal{A})$ and $\chi_A(x)\neq 0$.
\end{proof}

\begin{example}
If $X$ is a $T_1$ topological space (every one-pointic set is closed) then each singleton $\{x\}$ is in the Borel sets $\mathcal{B}(X)$ and thus the Borel sets form a locally finite $\sigma$-algebra on $X$.
\end{example}

\begin{remark}
Portion (b) of Theorem 4.3 is still valid if the separatedness hypothesis about $\mathcal{A}$ is dropped.
\end{remark}

\begin{example}
Here is an example of a separating $\sigma$-algebra on $\mathbb{N}$ containing no one-pointic set as a member. Let $\mathfrak{F}=\{\{1,2\},\{3,4\},\{5,6\},...,\{2n-1,2n\},....\}$. Then the smallest $\sigma$-algebra $\mathcal{A}$ on $\mathbb{N}$ containing $\mathfrak{F}$ is locally finite but no singleton is in $\mathcal{A}$. It is fairly easy to check that $\mathcal{A}$ consists of all (possibly empty) union of members of $\mathfrak{F}$.
\end{example}

\begin{theorem}
Let $(X,\mathcal{A})$ be locally finite. Then the following three statements are equivalent.

\begin{itemize}
\item[(1)] $\mathcal{A}$ is a finite $\sigma$-algebra on $X$. 
\item[(2)] $\mathcal{M}_F(X,\mathcal{A})=\mathcal{M}_\infty(X,\mathcal{A})$.
\item[(3)] $\mathcal{M}_\infty(X,\mathcal{A})$ is an ideal (possibly improper) of $\mathcal{M}(X,\mathcal{A})$.

\end{itemize}
\end{theorem}

\begin{proof}
If $\mathcal{A}$ is finite, then $\mathcal{M}_F(X,\mathcal{A})=\mathcal{M}_\infty(X,\mathcal{A})=\mathcal{M}^*(X,\mathcal{A})=\mathcal{M}(X,\mathcal{A})$. So assume that $\mathcal{A}$ is infinite. Then by Lemma 2.12 in \cite{ABS}, there is a countably infinite family $\{A_n\}_{n}$ of pairwise disjoint nonempty members of $\mathcal{A}$. For each $n\in\mathbb{N}$, let $x_n\in A_n$. Since $(X,\mathcal{A})$ is locally finite, there exists a finite subset $B_n$ of $A_n$ such that $x_n\in B_n\in \mathcal{A}$. Define $f:X\mapsto\mathbb{R}$ as follows: $f(B_n)=\frac{1}{n}$ for all $n\in\mathbb{N}$ and $f(X\setminus \displaystyle \bigcup_{n\in\mathbb{N}}B_n)=0$. Then $f\in\mathcal{M}(X,\mathcal{A})$ by pasting lemma. Also $f\in \mathcal{M}_\infty(X,\mathcal{A})$ but $f\notin \mathcal{M}_F(X,\mathcal{A})$, since the set on which $f$ is not zero is infinite. Define $g:X\mapsto\mathbb{R}$ as follows: $g(B_n)=n$ and $g(X\setminus \displaystyle \bigcup_{n=1}^\infty B_n)=0$. then $g\in\mathcal{M}(X,\mathcal{A}), fg=1$ on $\displaystyle \bigcup_{n=1}^\infty B_n$. Hence $fg\notin \mathcal{M}_\infty (X,\mathcal{A})$ as $\displaystyle \bigcup_{n=1}^\infty B_n$ is an infinite set. Thus $\mathcal{M}_\infty (X,\mathcal{A})$ is not an ideal of $\mathcal{M}(X,\mathcal{A})$.
\end{proof}

In the next result, we find out a condition on the $\sigma$-algebra $\mathcal{A}$ on $X$, both necessary and sufficient to make $\mathcal{M}_F(X,\mathcal{A})$ and $\mathcal{M}_\infty(X,\mathcal{A})$ trivial ideals of $\mathcal{M}(X,\mathcal{A})$.

\begin{theorem}
The following statements are equivalent.
\begin{itemize}
\item[(1)] $(X,\mathcal{A})$ is nowhere locally finite (i.e. no nonempty member of $\mathcal{A}$ is a finite set).
\item[(2)] $\mathcal{M}_F(X,\mathcal{A})=\{0\}$.
\item[(3)] $\mathcal{M}_\infty(X,\mathcal{A})=\{0\}$.
\end{itemize}
\end{theorem}

\begin{proof}
$(2)\Rightarrow (1):$ Suppose that $(1)$ is false. Then there is a finite nonempty $E\in\mathcal{A}$. The characteristic function $\chi_{E}\neq 0$ and is a member of $\mathcal{M}_F(X,\mathcal{A})$.

$(1)\Rightarrow (2)$ is trivial.

$(1)\Rightarrow (3):$ Suppose that $(3)$ is false. Then there is a $g\in\mathcal{M}_\infty(X,\mathcal{A})$, where $g\neq 0$. So there exists an $x_\circ\in X$ such that $\lvert g(x_\circ)\rvert=\lambda>0$. Then $\{x\in X:\lvert g(x)\rvert\geq \frac{\lambda}{2}\}$ is a nonempty finite set in $X$. Hence $(1)$ is false.
\end{proof}

\begin{example}
Here is an example of an infinite $\sigma$-algebra that has no nonempty finite sets.
For each $n\in \mathbb{N}$, let $p_n$ be the $n$-th prime. Let $\mathbb{N}_n=\{p_n^k:k\geq 1\}$, for each $n\geq 1$, let $\mathbb{N}_0=\mathbb{N}\setminus\cup_{n=1}^\infty\mathbb{N}_n$. Let $\mathcal{A}$ be the smallest $\sigma$-algebra on $\mathbb{N}$ containing all the $\mathbb{N}_n$ for $n\geq 0$. Then $\mathcal{A}$ is an infinite $\sigma$-algebra that has no nonempty finite sets.
\end{example}

We conclude this section by identifying some relevant topological nature of $\mathcal{M}_\infty(X,\mathcal{A})$ and $\mathcal{M}_F(X,\mathcal{A})$ as subspaces of the Banach algebra $\mathcal{M}^*(X,\mathcal{A})$ with supremum norm.

\begin{theorem}
\begin{itemize}
\item[(a)] $\mathcal{M}_\infty(X,\mathcal{A})$ is a closed subalgebra of $\mathcal{M}^*(X,\mathcal{A})$.
\item[(b)] $\mathcal{M}_F(X,\mathcal{A})$ is dense in $\mathcal{M}_\infty(X,\mathcal{A})$ in the uniform norm topology.
\end{itemize}
\end{theorem}

\begin{proof}
(a) Let $\lim\limits_{n\mapsto \infty}f_n=f$ in the Banach algebra $\mathcal{M}^*(X,\mathcal{A})$ where each $f_n\in\mathcal{M}_\infty(X,\mathcal{A})$. It is enough to show that $f\in \mathcal{M}_\infty(X,\mathcal{A})$. Choose $\epsilon>0$. Then there is a $k\in \mathbb{N}$ such that $\displaystyle \sup_{x\in X}\lvert f_k(x)-f(x)\rvert<\frac{\epsilon}{2}$. Hence $\{x\in X:\lvert f(x)\rvert\geq \epsilon\}\subseteq \{x\in X:\lvert f_k(x)\rvert\geq\frac{\epsilon}{2}\}$. Since $f_k\in \mathcal{M}_\infty(X,\mathcal{A})$, the later set is finite. It follows that $\{x\in X:\lvert f(x)\rvert \geq \epsilon\}$ is finite. Thus $f\in \mathcal{M}_\infty(X,\mathcal{A})$.

(b) Let $f\in\mathcal{M}_\infty(X,\mathcal{A})$
and $\epsilon>0$. Then the set $K=\{x\in X:\lvert f(x)\rvert\geq\epsilon\}$ is finite. It follows that the characteristic function $\chi_{K}\in\mathcal{M}_F(X,\mathcal{A})$ and hence $f\chi_K\in\mathcal{M}_F(X,\mathcal{A})$ as the later family of functions is already an ideal of the ring $\mathcal{M}(X,\mathcal{A})$. We check that for each $x\in X$, $\lvert(f\chi_K)(x)-f(x)\rvert\leq \epsilon$. Hence $\displaystyle \sup_{x\in X}\lvert(f\chi_K)(x)-f(x)\rvert\leq\epsilon$ i.e. $\|f\chi_K-f\|\leq\epsilon$.
\end{proof}

The following problem seems to be open:

\begin{question}
	Is $\mathcal{M}_\infty(X)$ equal to the intersection of all free maximal ideals of $\mathcal{M}^*(X,\mathcal{A})?$
\end{question}

\textbf{Acknowledgments} The authors take the pleasure to acknowledge Professor Alan Dow for suggesting the example 2.14.

\bibliographystyle{plain}

\end{document}